\documentclass[a4paper,10pt]{article}
\usepackage{amsfonts}
\usepackage{amssymb}
\usepackage{amsmath}
\usepackage{amsthm}

\usepackage{fancyhdr}

\usepackage{graphicx}
\usepackage[colorlinks=true,dvipdfm]{hyperref}

\newcommand{\keywords}[1]{\textbf{\textit{Keywords:}} #1}

\newtheorem{lem}{Lemma}
\newtheorem{theorem}{Theorem}
\newtheorem{prop}{Proposition}

\theoremstyle{definition}
\newtheorem{defin}{Definition}
\newtheorem*{algorithm}{Algorithm}

\theoremstyle{remark}
\newtheorem{rem}{Remark}

\newtheorem*{example}{Example}
\begin{document}

\title{Partial Classification of Lorenz Knots: Syllable Permutations of Torus Knots Words}

\author{Paulo Gomes\thanks{\'Area Departamental de Matem\'atica, Instituto Superior de Engenharia de Lisboa, e-mail: pgomes@adm.isel.pt}, Nuno Franco\thanks{CIMA-UE and Departamento de Matem\'atica, Universidade de \'Evora, e-mail: nmf@uevora.pt} and Lu\'is Silva\thanks{CIMA-UE and \'Area Departamental de Matem\'atica, Instituto Superior de Engenharia de Lisboa, e-mail: lfs@adm.isel.pt}}

\maketitle

\begin{abstract}
  We define families of aperiodic words associated to Lorenz knots that arise naturally as syllable permutations of symbolic words corresponding to torus knots. An algorithm to construct symbolic words of satellite Lorenz knots is defined. We prove, subject to the validity of a previous conjecture, that Lorenz knots coded by some of these families of words are hyperbolic, by showing that they are neither satellites nor torus knots and making use of Thurston's theorem. Infinite families of hyperbolic Lorenz knots are generated in this way, to our knowledge, for the first time. The techniques used can be generalized to study other families of Lorenz knots.
\end{abstract}


\section{Introduction}
\label{sec:intro}

\subsection*{Lorenz knots}
\label{sec:lorknots}

\par

\emph{Lorenz knots} are the closed (periodic) orbits in the Lorenz system \cite{Lorenz63}

\begin{align}
  \label{eq:lorsys}
  x' &= -10 x +10 y \nonumber \\
  y' &= 28 x -y -xz \\
  z' &= -\frac{8}{3} z +xy \nonumber 
\end{align}
while \emph{Lorenz links} are finite collections of (possibly linked) Lorenz knots.

The systematic study of Lorenz knots and links was made possible by the introduction of the \emph{Lorenz template} or knot-holder, which is a branched 2-manifold introduced by Williams \cite{Williams77},\cite{Williams79}.  It is built from one joining chart, one splitting chart and an expanding semi-flow defined on them (Figs. \ref{fig:joinsplit} and \ref{fig:lortemp}). It was first conjectured by Guckenheimer and Williams and later proved through the work of Tucker and Ghys that every knot and link in the Lorenz system can be projected into the Lorenz template. Birman and Williams made use of this result to investigate Lorenz knots and links \cite{Birman83}. For a review on Lorenz knots and links, see also \cite{Birman11}.

A $T(p,q)$ torus knot is (isotopic to) a curve on the surface of an unknotted torus $T^2$ that intersects a meridian $p$ times and a longitude $q$ times. Birman and Williams \cite{Birman83} proved that every torus knot is a Lorenz knot.

A satellite knot is defined as follows: take a nontrivial knot $C$ (companion) and nontrivial knot $P$ (pattern) contained in a solid unknotted torus $T$ and not contained in a $3-ball$ in $T$. A satellite knot is the image of $P$ under an homeomorfism that takes the core of $T$ onto $C$.

A knot is  hyperbolic if its complement in $S^3$ is a hyperbolic $3-manifold$. Thurston \cite{Thurston82} proved that a knot is hyperbolic \emph{iff} it is neither a satellite knot nor a torus knot. One of the goals in the study of Lorenz knots has been their classification into \emph{hyperbolic} and \emph{non-hyperbolic}, possibly further distinguishing torus knots from satellites.  Birman and Kofman \cite{Birman09} listed hyperbolic Lorenz knots taken from a list of the simplest hyperbolic knots. In a previous article we generated and tested for hyperbolicity, using the program \emph{SnapPy}, families of Lorenz knots that are a generalization of some of those that appear in this list, which led us to conjecture that the families tested are hyperbolic \cite{Gomes13}.

The first-return map induced by the semi-flow on the \emph{branch line} (the horizontal line in Fig. \ref{fig:lortemp} where the two branches of the joining chart meet) is called the \emph{Lorenz map}. If the branch line is mapped onto $[-1,1]$, then the Lorenz map $f$ becomes a one-dimensional map from $[-1,1] \setminus \{0\}$ onto $[-1,1]$, with one discontinuity at $0$ and stricly increasing in each of the subintervals $[-1,0[$ and $]0,1]$ (Fig. \ref{fig:lormap}).

\begin{figure}
  \centering
  \includegraphics[scale=.50]{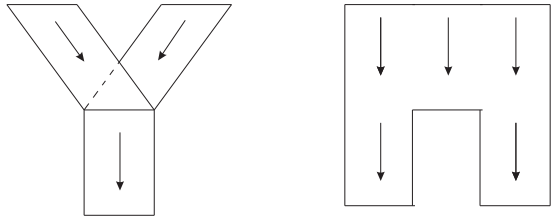}
  \caption{Joining (left) and splliting (right) charts}
  \label{fig:joinsplit}
\end{figure}

\begin{figure}
  \centering
  \includegraphics[scale=.28]{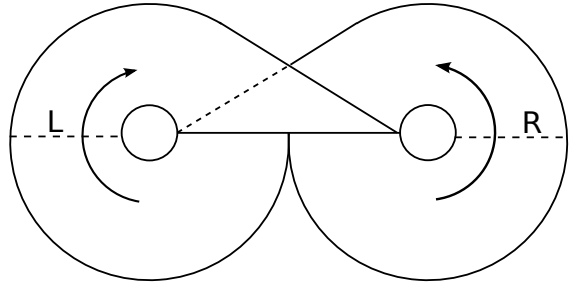}
  \caption{The Lorenz template}
  \label{fig:lortemp}
\end{figure}

\begin{figure}
  \centering
  \includegraphics[scale=1.0]{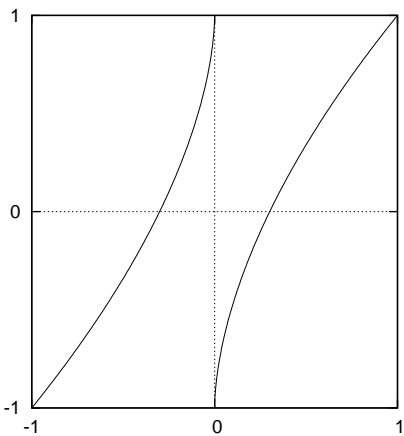}
  \caption{Lorenz map}
  \label{fig:lormap}
\end{figure}

\subsection*{Lorenz braids}
\label{sec:lorbraids}

If the Lorenz template is cut open along the dotted lines in Fig. \ref{fig:lortemp}, then each knot and link on the template can be obtained as the closure of an open braid on the cut-open template, which will be called the \emph{Lorenz braid} associated to the knot or link (\cite{Birman83}). These \emph{Lorenz braids} are simple positive braids (our definition of positive crossing follows Birman and is therefore opposed to an usual convention in knot theory). Each Lorenz braid is composed of $n=p+q$ strings, where the set of $p$ left or $L$ strings cross over at least one (possibly all) of the $q$ right strings, with no crossings between strings in each subset. These sets can be subdivided into subsets $LL$, $LR$, $RL$ and $RR$ according to the position of the startpoints and endpoints of each string. An example of a Lorenz braid is shown in Fig. \ref{fig:lorbraid}, where we adopt the convention of drawing the overcrossing ($L$) strings as thicker lines than the undercrossing ($R$) strings. This convention will be used in other braid diagrams.

The \emph{braid group on n strings} $B_n$ is given by the presentation
$$B_n=\left\langle \sigma_1,\sigma_2,\ldots,\sigma_{n-1} \left\lvert \begin{array}{ll} \sigma_i\sigma_j=\sigma_j\sigma_i & (|i-j|\geq2\\ \sigma_i\sigma_{i+1}\sigma_i=\sigma_{i+1}\sigma_i\sigma_{i+1} & (i=1,\ldots,n-2) \end{array} \right. \right\rangle $$
where the \emph{generator} $\sigma_i$ exchanges the endpoints of strings $i$ and $i+1$ with string $i$ crossing over string $i+1$. In particular, all Lorenz braids can be expressed as products of these generators.

Each Lorenz braid $\beta$ is a simple braid, so it has an associated permutation $\pi$. This permutation has only one cycle \emph{iff} it is associated to a knot, and has $k$ cycles if it is associated to a link with $k$ components (knots).

\begin{figure}
  \centering
  \includegraphics[scale=1.0]{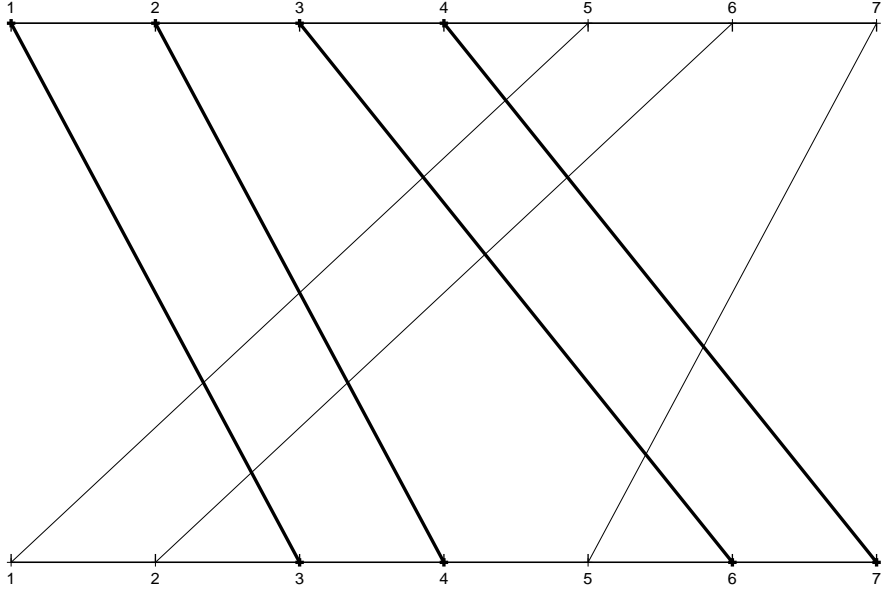}
  \caption{A Lorenz braid}
  \label{fig:lorbraid}
\end{figure}

\subsection*{Symbolic dynamics for the Lorenz map}
\label{sec:symbdynlor}

Let $f^j=f \circ f^{j-1}$ be the $j$-th iterate of the Lorenz map $f$ anf $f^0$ be the identity map. We define the \emph{itinerary} of a point $x$ under $f$ as the symbolic sequence $(i_f(x))_j$, $j=0,1,\ldots$ where
$$(i_f(x))_j=\left\{\begin{array}{lll} L & \mathrm{ if } & f^j(x)<0\\
                                       0 & \mathrm{ if } & f^j(x)=0\\
                                       R & \mathrm{ if } & f^j(x)>0. \end{array} \right.$$

 The itinerary of a point in $[-1,1]\setminus \{0\}$ under the Lorenz map can either be an infinite sequence in the symbols $L,R$ or a finite sequence in $L,R$ terminated by a single symbol $0$ (because $f$ is undefined at $x=0$).  The \emph{length} $|X|$ of a finite sequence $X = X_0 \ldots X_{n-1}0$ is $n$, so it can be written as $X = X_0 \ldots X_{|X|-1}0$. A sequence $X$ is periodic if $X=(X_0 \dots X_{p-1})^{\infty}$ for some $p>1$. If $p$ is the least integer for which this holds, then $p$ is the (least) period of $X$.

The space $\Sigma$ of all finite and infinite sequences can be ordered in the lexicographic order induced by $L < 0 < R$: given $X, Y \in \Sigma$, let $k$ be the first index such that $X_k \neq Y_k$. Then $X<Y$ if $X_k < Y_k$ and $Y < X$ otherwise.

The \emph{shift map} $s:\Sigma\setminus\{0\} \to \Sigma$ is defined as usual by $s(X_0X_1 \ldots)=X_1 \ldots$ (it just deletes the first symbol). From the definition above, an infinite sequence $X$ is periodic \emph{iff} there is $p>1$ such that $s^p(X)=X$. In order to symplify the notation, we will also define the shift operator on finite aperiodic words. Given a $p$-periodic sequence $X=(X_0X_1 \ldots X_{p-1})^{\infty}$ where $w=X_0X_1 \ldots X_{p-1}$ is a finite aperiodic word of length $p$, we define $s(X_0X_1 \ldots X_{p-1})=X_1 \ldots X_{p-1}X_0$. Then $s(X)=(X_1 \ldots X_{p-1}X_0)^\infty=(s(X_0X_1 \ldots X_{p-1}))^\infty$.  This shift operator defined on finite words has an inverse $s^{-1}$, defined by $s(X_0X_1 \ldots X_{p-1})=X_{p-1}X_0X_1 \ldots X_{p-2}$. The sequence \linebreak$w,s(w),\ldots,s^{p-1}(w)$ will also be called the \emph{orbit} of $w$ and a word in the orbit of $w$ will be generally called a shift of $w$.

A (finite or infinite) sequence $X$ is called \emph{L-maximal}
 if $X_0=L$ and for $k>0$, $X_k = L \Rightarrow s^k(X)\leq X$, and \emph{R-minimal} if $X_0=R$ and for $k>0$, $X_k = R \Rightarrow X \geq s^k(X)$. An infinite periodic sequence $(X_0 \ldots X_{n-1})^{\infty}$ with least period $n$ is L-maximal (resp. R-minimal) if and only if the finite sequence $X_0 \ldots X_{n-1}0$ is L-maximal (resp. R-minimal). Given an aperiodic word $w=X_0 \ldots X_n$, we say that $k$ is the maximal (resp. minimal) position in $w$ if $s^k(w)$ is $L$-maximal (resp. $R$-minimal).  

For a finite word $w$, $n_L$ and $n_R$ will denote respectively the number of $L$ and $R$ symbols in $w$, and $n=n_L+n_R$ the lenghth of $w$. Each finite aperiodic word is associated to a Lorenz braid (whose closure is a Lorenz knot), which can be obtained through the following procedure: given a word $w$ of length $n$, order the successive shifts $s(w),s^2(w),\ldots,s^n(w)=w$ lexicographically and associate them to startpoints and endpoints in the associated Lorenz braid, with points corresponding to words starting with $L$ lying on the left half and points corresponding to words starting with $R$ on the right half. Each string in the braid connects the startpoint corresponding to $s^k(w)$ to the endpoint corresponding to $s^{k+1}(w)$. Fig. \ref{fig:lorbraid-word} exemplifies this procedure for $w=LRRLR$.

\begin{figure}
  \centering
  \includegraphics[scale=1.0]{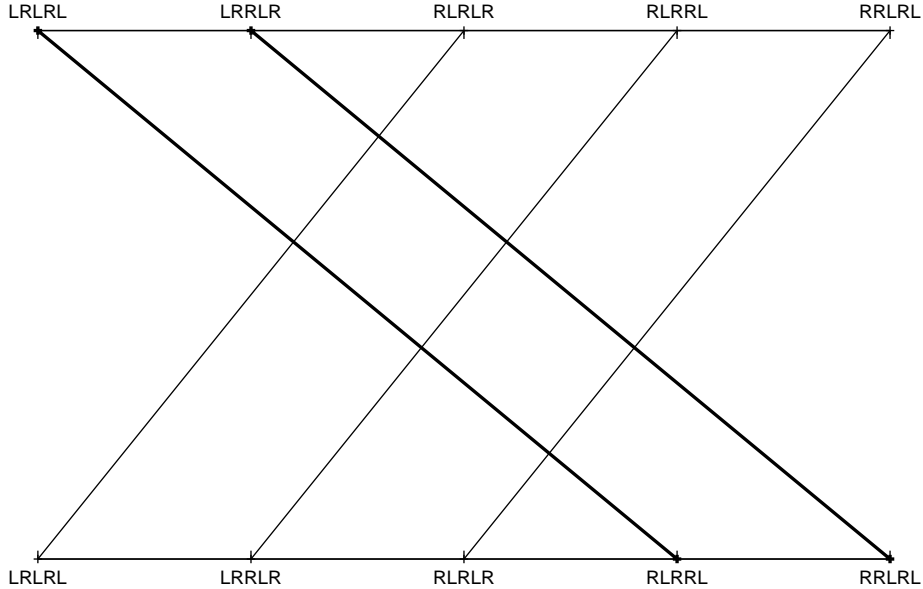}
  \caption{Lorenz braid corresponding to $w=LRRLR$}
  \label{fig:lorbraid-word}
\end{figure}

Each periodic orbit of the flow has a unique corresponding orbit in the Lorenz map, which in turn corresponds to the cyclic permutation class of one aperiodic word in the symbols $L,R$ as clearly all periodic sequences resulting from shifting a given sequence represent points in the same periodic orbit (see \cite{Birman83}).

The \emph{crossing number} is the smallest number of crossings in any diagram of a knot $K$. The \emph{braid index} is the smallest number of strings among braids whose closure is $K$.

The \emph{trip number} $t$ is the number of syllables (subwords of type $L^aR^b$ with maximal length) in an aperiodic word. The trip number of a Lorenz link is the sum of the trip numbers of its components. Franks and Williams \cite{Franks87}, followed by Waddington \cite{Waddington96}, proved that the braid index of a Lorenz knot is equal to its trip number. This result had previously been conjectured by Birman and Williams \cite{Birman83}, who defined a minimal $t-$braid for Lorenz links \cite{Birman09}:

The minimal braid of a Lorenz link corresponding to a Lorenz braid $\beta$ is given by
$$\Delta^2 \prod_{i=1}^{t-1}\, (\sigma_1 \ldots \sigma_i)^{n_i} \prod_{i=t-1}^1 (\sigma_{t-1}\ldots \sigma_i)^{m_{t-i}}$$
where $\Delta^2$ is the full-twist in $t$ strings, and, denoting by $\pi$ the permutation associated with the Lorenz braid, the exponents are
$$ \begin{array}{l} n_i=\#\{j: \pi(j) - j = i+1 \land \pi(j)<\pi^2(j)\}\\ m_i = \#\{j: j - \pi(j) = i+1 \land \pi(j)>\pi^2(j)\} \end{array} $$ 

\subsection*{Goal and plan}
\label{sec:plan}

The purpose of this article is the classification of knots corresponding to syllable permutations of the standard word of a torus knot, that is, words composed of the same number of $LR^k$ and $LR^{k+1}$ syllables as the torus knot standard word, defined below, arranged in a different order.

In Section \ref{sec:torsylperm} we characterize the symbolic words obtained by permuting the sylalbles of words corresponding to torus knots and their associated Lorenz braids. Section \ref{sec:Lorenzsat} deals with braids of Lorenz satellite knots and an algorithm to find their corresponding words, also proving the permuted words cannot be obtained this way. Finally, in Section \ref{sec:Lorenztor} we show that in each set of syllable permutations of a torus knot word, there is at most one word that can possibly correspond to another torus knot, and prove that for some classes of these sets there is no such word.

\section{Torus knots and syllable permutations}
\label{sec:torsylperm}

As mentioned above, all torus knots are Lorenz knots. The torus knot $T(p,q)$ is the closure of a Lorenz braid in $n=p+q$ strings, with $p$ left or $L$ strings that cross over $q$ right or $R$ strings, such that each $L$ string crosses over all the $R$ strings. The Lorenz braid of a torus knot thus has the maximum number of crossings ($pq$) for a Lorenz braid with $p$ $L$ strings and $q$ $R$ strings. Since $T(q,p)=T(p,q)$ we will only consider torus knots $T(p,q)$ with $p<q$. 
The structure of the Lorenz braid of a torus knot $T(p,q)$ ($p<q$) is sketched in Fig. \ref{fig:torbraid}, where only the first and last $L$ strings and some of the $R$ strings are drawn. The remaining $L$ ($R$) strings are parallel to the $L$ ($R$) strings shown.

\begin{figure}
  \centering
  \includegraphics{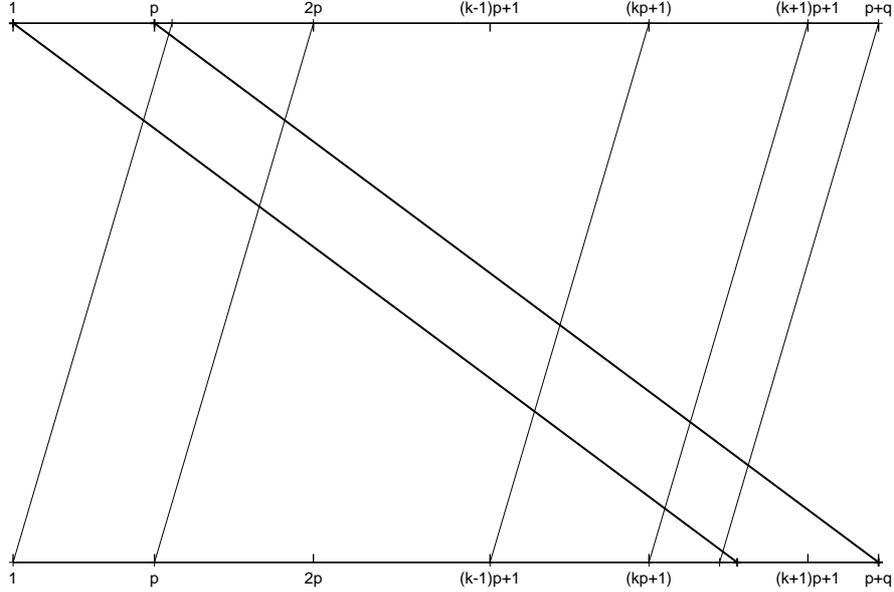}
  \caption{Lorenz braid of $T(p,q) (p<q)$}
  \label{fig:torbraid}
\end{figure}

Lorenz knots corresponding to orbits in the Lorenz template which are represented by evenly distributed words in the alphabet $\{L,R\}$ are torus knots \cite{Birman83}.  Also, given a torus knot $T(p,q)$ there is an evenly distributed word with $n_L=p$, $n_R=q$, that represents it. There is thus a bijection between torus knots and cyclic permutation classes of evenly distributed words. We will call the $L-maximal$ word that represents $T(p,q)$ the \emph{standard word} $w(p,q)$ for $T(p,q)$. These are also the words generated in the symbolic Farey tree \cite{Franco08}. Let $k>0$ be the quotient of the integer division of $q$ by $p$ and $0<r<p$ the remainder ($q=kp+r$). Then $w(p,q)$ has $p$ syllables: $r$ $LR^{k+1}$ syllables and $p-r$ $LR^k$ syllables, evenly distributed. This aperiodic word is unique for each $p,q$ satisfying the conditions above.

\begin{lem}
  Any word resulting from permuting the syllables of the standard word $w(p,q)$ is aperiodic and therefore corresponds to a Lorenz knot with the same braid index.
\end{lem}
\begin{proof}
  Let $T(p,q)$, $p<q$ be a torus knot and $w$ a syllable permutation of $w(p,q)$. Then $n_L(w)=p$ and $n_R(w)=q$. If there was a subword $v$ of $w$ such that $w=v^j$ for some $j>1$, then we would have $p=n_L(w)=jn_L(v)$ and $q=n_R(w)=jn_R(v)$ and $p,q$ wouldn't be relatively prime.
\end{proof}

\begin{defin}
  We define $P(p,q)$ as the set of $L$-maximal words resulting from permutations of syllables of the standard word $w(p,q)$ for $T(p,q)$.
\end{defin}

\begin{rem}
  Any word resulting from the permutation of syllables of the standard word for $T(p,q)$ is a shift of some word in $P(p,q)$ and is therefore equivalent to that word, \emph{i.e.} it corresponds to the same Lorenz knot.
\end{rem}

\begin{rem}
  The set of distinct permutations of $p$ syllables, where respectively $r$ and $p-r$ are identical, contains $\frac{p!}{r!(p-r)!}$ words. This set contains, for each $L$-maximal word $w$, $p$ words that are shifts of $w$ (including itself). To keep only the $L$-maximal words we must therefore further divide by $p$. Each set $P(p,q)$, $q=kp+r$ has therefore exactly $\frac{(p-1)!}{r!(p-r)!}$ words.
\end{rem}

\begin{rem}\label{rem:p>4}
  If $2 \leq p \leq 4$, then $P(p,q)$ contains only the standard word for $T(p,q)$: $P(p,q)=\{w(p,q)\}$.
\end{rem}

\begin{rem}\label{rem:r=1,p-1}
  If $q=kp+r$ and $r=1$ or $r=p-1$, then the only possible $L$-maximal word in $P(p,q)$ is $LR^{k+1}(LR^k)^{p-1}$ ($r=1$) or $(LR^{k+1})^{p-1}LR^k$ ($r=p-1$), so in both cases $P(p,q)$ contains only the standard word for $T(p,q)$.
\end{rem}

In what follows, we therefore assume, whenever necessary, $p>4$ and $1<r<p-1$.

Next we find bounds for the number of crossings in the Lorenz braid and in the minimal Birman-Williams (BW) braid. For a knot which is the closure of a positive braid, the genus $g$ is related to the number of crossings $c$ and the number of strings $n$ by \cite{Birman83}
\begin{equation}
  \label{eq:1}
  2g=c-n+1
\end{equation}

In the symbolic words under study there are no consecutive $L$ symbols, so for each knot in $P(p,q)$ the trip number is $t=n_L=p$. The minimal BW braid \cite{Birman83} is a $p$-string braid and the number of crossings (the length of the braid word) in the BW braid is the crossing number of the knot, because in Eq. \ref{eq:1}, if $c$ takes the minimum value, then $n$ must also be minimum, as the genus $g$ is a knot invariant. 

We start by investigating the structure of the braids that correspond to the words with permuted syllables. Let $w$ be an aperiodic word resulting from permuting the syllables of the standard word $w(p,q)$. Note that the shifts of $w$ can be grouped and ordered lexicographically: $$LR^kL \dotso <LR^{k+1}\dotso <RL\dotso<R^2L\dotso<\ldots<R^kL\dotso<R^{k+1}L \dotso$$

Also, there are:
\begin{itemize}
\item exactly $p$ words (and corresponding strings) with the form $LR\dotso$, of which $r$ are $LR^{k+1}\dotso$ and $p-r$ are $LR^kL\dotso$;
\item $p$ words with each of the forms  $RL\dotso$, $R^2L\dotso$, $\dotso$ , $R^kL\dotso$;
\item $r$ words with the form $R^{k+1}L$.
\end{itemize}

Moreover:
\begin{enumerate}
\item All the words of type $LR\dotso$ are shifts some word of type $RL\dotso$, and the shift operator preserves lexicographical order.

For $1\leq i \leq k-1$, the words of type $R^iL\dotso$ are the shifts of $R^{i+1}L\dotso$ words, with order preserved.
\item The $p$ words of type $R^kL\dotso$ are shifts of the $r$ $R^{k+1}L\dotso$ and the $p-r$ $LR^kL\dotso$ words, the shift operator preserving the order within each subset.

The least word with the form $R^kL\dotso$, $w_l$, is the shift of the least $R^{k+1}L\dotso$ word. To see this, assume that $w_l$ is the shift of a word of type $LR^kL\dotso$. This word has a preimage under the iterated shift operator with the form $R^kLR^kL\dotso$, resulting from appending an $R^kL$ syllable to the beginning of $w_l$, and therefore $R^kLR^kL\dotso < w_l=R^kL\dotso$, which contradicts the fact that $w_l$ is the least $R^kL\dotso$ word.

Analogously, the greatest $R^kL\dotso$ word, $w_g$, is the shift of the greatest $LR^kL\dotso$ word. Assume that $s^{-1}(w_m)$ is of type $R^{k+1}L\dotso$; $w_m$ has the form $R^kLR^{k+1}L\dotso$ (otherwise it wouldn't be the greatest $R^kL\dotso$), so $s^{-1}(w_m)$ must have the form $R^{k+1}LR^{k+1}L\dotso$. This word has a preimage under the shift with the form $R^kLR^{k+1}LR^{k+1}\dotso$ and is greater than $w_g$ (it has one more syllable $R^{k+1}$ immediately following the first $R^kL$), which contradicts the fact that $w_g$ is the greatest $R^kL\dotso$ word.
\item The $r$ words of type $R^{r+1}L\dotso$ are the shifts of the $r$ $LR^{k+1}\dotso$ words, with order preserved.
\end{enumerate}

The Lorenz braids corresponding to the permuted words thus have the structure sketched in Fig. \ref{fig:2}. The dotted lines and enclosed shaded regions represent two sets of strings and which both have their endpoints in the range $\{kp+2, \ldots, (k+1)p-1\}$, so that the strings from the first set with startpoints $1, \ldots, p-r-1$ will possibly cross over the strings with startpoints $(k+1)p+2, (k+1)p+r=p+q$,  preserving the order in each set of strings.

\begin{figure}
  \centering
  \includegraphics{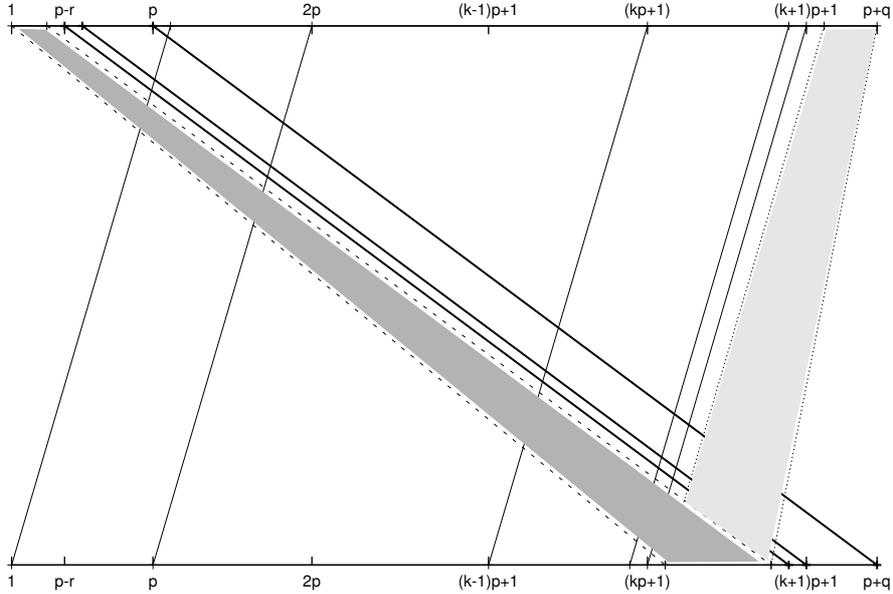}
  \caption{Lorenz braid corresponding to a syllable permuted word}
  \label{fig:2}
\end{figure}

The Lorenz braids corresponding to the permuted words have a number of crossing points $c$ that satisfies

\begin{equation}
  \label{crossings}
  (p-r-1)(kp+1)+(r+1)(kp+r)\leq c \leq p(kp+r)=pq.
\end{equation}

For the standard word $w(p,q)$, since all the $p$ $L$ strings cross over the $q=kp+r$ $R$ strings, we have $c=p(kp+r)$. 

\begin{lem}\label{order}
For the word $(LR^{k+1})^r(LR^k)^{p-r}$, the corresponding braid has the structure illustrated in Fig. \ref{fig:3}, with exactly $(p-r-1)(kp+1)+(r+1)(kp+r)$ crossings. Since the $L$ strings have the leftmost possible endpoints and the $R$ strings the maximum possible endpoints, this braid has the minimum number of crossings in $P(p,q)$ and $(LR^{k+1})^r(LR^k)^{p-r}$ is the only word in $P(p,q)$ corresponding to this number of crossings.
\end{lem}

\begin{proof}
The only endpoints that can change from one permuted braid corresponding to a word $w$ in $P(p,q)$ are those corresponding to shifts of $w$ with the form $R^kL\ldots$. These can be ordered lexicographically as follows:

\begin{align*}
  s\left(R^{k+1}(LR^k)^{p-r}(LR^{k+1})^{r-1}L \right) &= R^k(LR^k)^{p-r}(LR^{k+1})^{r-1}LR < \\
< s\left((LR^k)^{p-r}(LR^{k+1})^r \right) &= R^k(LR^k)^{p-r-1}(LR^{k+1})^rL < \\
< s\left((LR^k)^{p-r-1}(LR^{k+1})^r \right) &= R^k(LR^k)^{p-r-2}(LR^{k+1})^rLR^kL < \ldots\\
\ldots < s\left(LR^k(LR^{k+1})^r(LR^k)^{p-r-1} \right) &= R^k(LR^k)(LR^{k+1})^r(LR^k)^{p-r-2}L<\\
< s\left(R^{k+1}(LR^{k+1})(LR^k)^{p-r}(LR^{k+1})^{r-2}L\right) &= R^k(LR^{k+1})(LR^k)^{p-r}(LR^{k+1})^{r-2}LR<\ldots\\
\ldots < s\left(R^{k+1}(LR^{k+1})^{r-1}(LR^k)^{p-r}L \right) &= R^k(LR^{k+1})^{r-1}(LR^k)^{p-r}LR <\\
< s\left(LR^k(LR^{k+1})^r(LR^k)^{p-r-1} \right) &= R^k(LR^{k+1})^r(LR^k)^{p-r-1}L
\end{align*}

Identifying the words between brackets on the left with the startpoints of the corresponding strings and the ordered words on the right with the matching endpoints, we get the braid represented  in Fig. \ref{fig:3}.
\end{proof}

\begin{figure}
  \centering
  \includegraphics{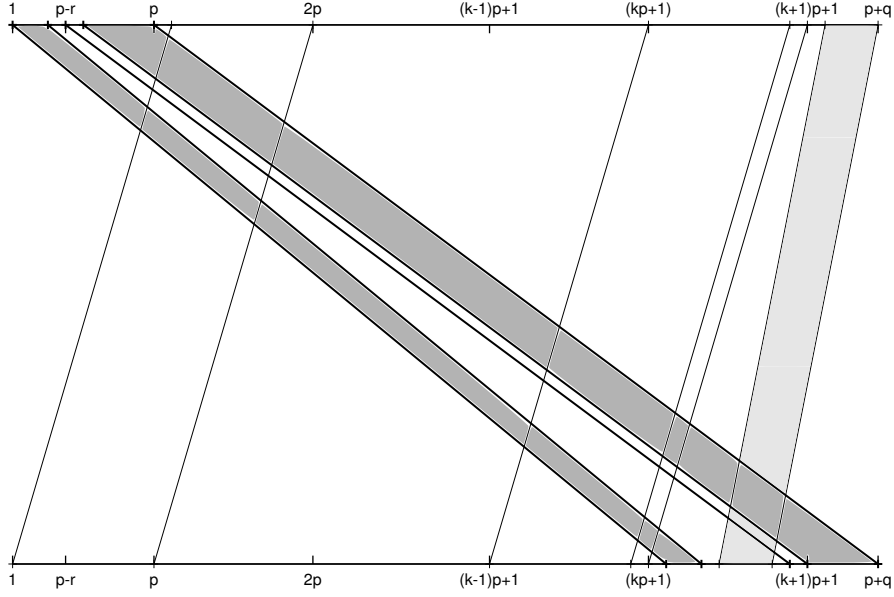}
  \caption{Lorenz braid of $(LR^{k+1})^r(LR^k)^{p-r}$}
  \label{fig:3}
\end{figure}

The Lorenz braid of $(LR^{k+1})^r(LR^k)^{p-r}$ thus has exactly $(p-r-1)(kp+1)+(r+1)(kp+r)$ crossing points. This is the only braid in $P(p,q)$ with this number of crossings.

Using Eq. \ref{eq:1} we can now find bounds for the genus $g$ of knots that are the closure of braids in $P(p,q)$.

\begin{equation}
  \label{eq:genus}
  kp(p-1)+r(r-1) \leq 2g \leq (p-1)(kp+r-1)=(p-1)(q-1)
\end{equation}

Again, the minimum corresponds to $(LR^{k+1})^r(LR^k)^{p-r}$ while the maximum is the double of the torus knot genus and therefore corresponds to the standard word $w(p,q)$.

\section{Lorenz satellite knots}
\label{sec:Lorenzsat}

El-Rifai studied Lorenz knots which are satellites of Lorenz knots \cite{Elrifai88} \cite{Elrifai99}. His construction of a satellite knot or link can be interpreted in terms of Lorenz braids, as follows:

\begin{itemize}
\item Take three Lorenz braids $A$ (left pattern), $B$ (right pattern) and $C$ (companion) whose closures are Lorenz knots, such that $n_R(A)=n_L(B)=k$.
\item Inflate the braid $C$ by replacing each string with $k$ parallel strings, thus obtaining an intermediate braid corresponding to a link with $k$ components.
\item Add $n_L(A)$ vertical strings to the left and $n_R(B)$ vertical strings to the right (corresponding to identity permutations of these strings). The number of strings of the resulting braid is $n=n_L(A)+n_R(B)+kn(C)$, where $n(C)$ is the number of strings in $C$.
\item Concatenate this braid with the $n$-braid obtained by putting $A$ on the left, $B$ on the right and extend the remaining strings in $C$ with vertical strings, between $A$ and $B$.
\end{itemize}

The procedure is illustrated in Fig. \ref{fig:satbraid2} for braids $A$, $B$, $C$ corresponding to the words $w_A=LRRLR$, $w_B=LRRLRLR$ and $w_C=LRRLR$, respectively.

The construction can also be carried out with only one pattern braid $A$ or $B$. In that case, take the identity $k$-braid for $B$ ($A$) and assume $n_R(B)=0$ ($n_L(A)=0$) in the third step.

\begin{figure}
  \centering
  \includegraphics{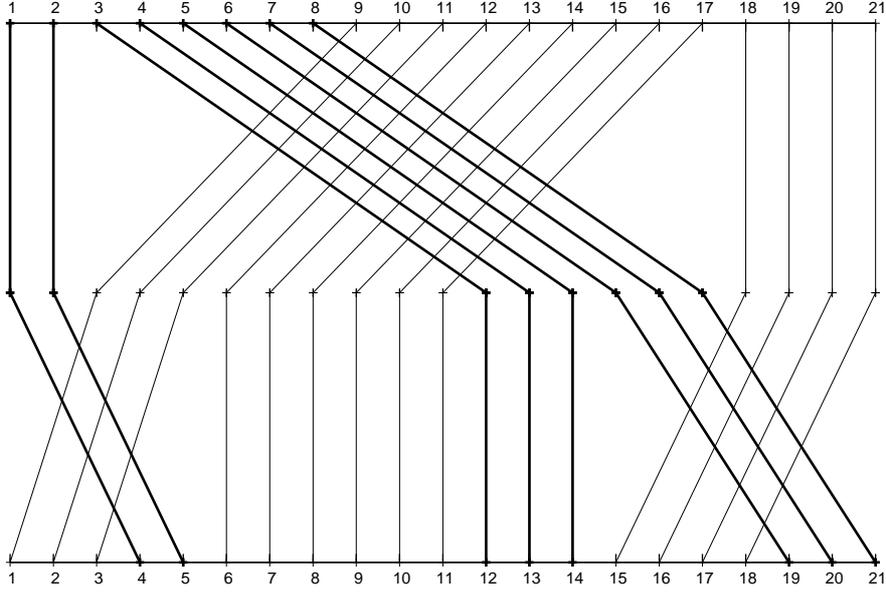}
  \caption{Satellite for $A=LRRLR$, $B=LRRLRLR$ and $C=LRRLR$}
  \label{fig:satbraid2}
\end{figure}

\begin{rem}
  The closure of the resulting braid can be a knot or a link. The braids $A$ and $B$ induce a permutation of the inflated strings of $C$ which is the product of permutations associated to $A$ and $B$. If this permutation has more than one cycle, then the resulting braid is a link. An example is given below, with $A=LRRLR$, $B=LRLRL$ and $C=LRRLR$ (Fig. \ref{fig:satbraid1}).
\end{rem}

\begin{figure}
  \centering
  \includegraphics{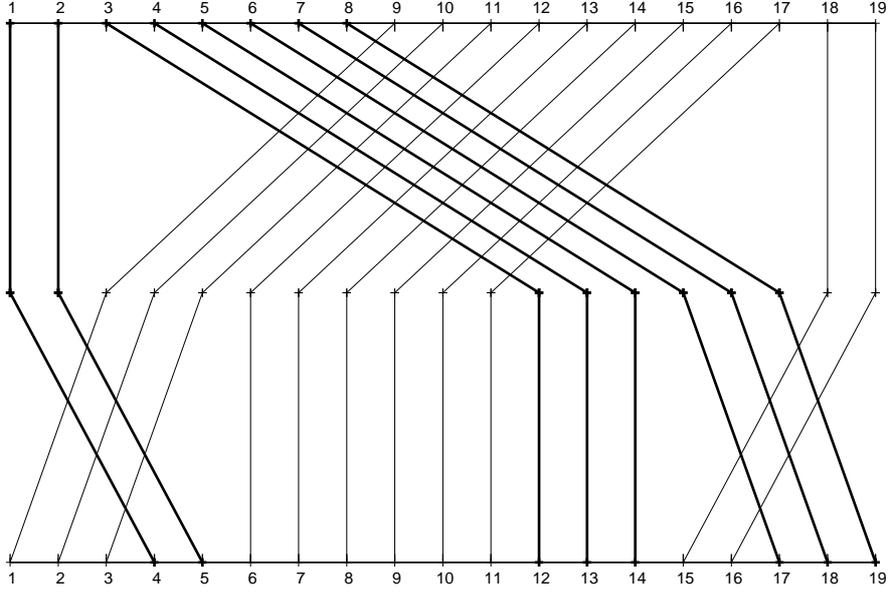}
  \caption{Link for $A=LRRLR$, $B=LRLRL$ and $C=LRRLR$}
  \label{fig:satbraid1}
\end{figure}

In order to obtain the aperiodic word of the satellite knot that is the closure of the Lorenz braid constructed as above, we start by defining permutations on the $L$ and $R$ points, respectively, of a Lorenz braid whose closure is a Lorenz knot, or equivalently on words starting respectively with $L$ or $R$ in the symbolic orbit associated to the braid.

\begin{defin}\label{def:LRperm}
  Let $b$ be a Lorenz $n$-braid that closes to a knot, $\pi(b)$ the permutation associated to $b$ and $w$ the corresponding $L$-maximal aperiodic symbolic word. We label the $n_L$ words which start with $L$ in the orbit of $w$ as $l_1,\dots,l_{n_L}$ and the $n_R$ words starting with $R$ as $r_1,\dots,r_{n_R}$, such that, under the lexicographic order, $l_1 < \dots < l_{n_L}$ and $r_1 < \dots < r_{n_R}$.

We define the $L$-permutation associated to $b$, $\pi_L$ for $i=1,\dots,n_L$ as $\pi_L(i) = j$ if $l_j$ is the first shifted word that starts with an $L$ in the orbit $s(l_i),\ldots,s^n(l_i)$. The $R$-permutation $\pi_R$ is defined for $i=1,\dots,n_R$ as $\pi_R(i)=j$ if $r_j$ is the first shifted word that starts with an $R$ in the orbit $s(r_i),\ldots,s^n(r_i)$. Maps $\pi_L$ and $\pi_R$ are therefore the first return maps on the sets of $L$-started and $R$-started words, respectively. 
\end{defin}

\begin{example}
  For $w=LRRLR$, the shifted words, ordered lexicographically, are $l_1=LRLRR$, $l_2=LRRLR$, $r_1=RLRLR$, $r_2=RLRRL$, $r_3=RRLRL$, so $\pi_L(1)=2$, $\pi_L(2)=1$, $\pi_R(1)=2$, $\pi_R(2)=3$, $\pi_R(3)=1$.
\end{example}

\begin{prop}
  The Lorenz braid constructed as above from Lorenz braids $A$, $B$ and $C$, with $n_R(A)=n_L(B)$, is associated to a cyclic permutation and therefore closes to a knot iff $\pi_R(A) \pi_L(B)$ is cyclic (has only one cycle). If only one braid $A$ or $B$ is used in the construction, then the closure of the resulting Lorenz braid is always a Lorenz knot.
\end{prop}

\begin{proof}
  Let $p_A=n_L(A)$, $q_B=n_R(B)$ and $n_C=n(C)$. Let $\pi$ be the permutation corresponding to the satellite braid constructed as above. For simplicity, denote $\pi_R(A)$ and $\pi_L(B)$ by $\pi_R$ and $\pi_L$, respectively. The orbit of a point $p$ under the permutation $pi$ is $\left( p,\pi(p),\ldots,\pi^{-1}(p)\right)$. The permutation $\pi$ is cyclic \emph{iff} given a point $p$, its orbit contains all points. We will subdivide the points in the satellite braid in te following sets:
  \begin{itemize}
  \item $A$ contains the first $p_A+k$ points, corresponding to the original braid $A$;
  \item $B$ contains the last $q_B+k$ points, corresponding to the original braid $B$;
  \item The remaing points (in the central part) are subdivided in $k$ sets $C_i$, $i=1,\ldots,k$ according to the component of the inflated braid they belong to.
  \end{itemize}
  Note that:
\begin{itemize}
\item If $p \in C_i$, then either $\pi(p) \in C_i$ or $\pi(p) \in A$ or $\pi(p) \in B$;
\item If $p \in A$ ($p \in B$) then the first point in the orbit of $p$ not in $A$ ($B$) will be in one of the sets $C_i$.
\item If $p \in C_i$ and $\pi(p) \in A$ ($\pi(p) \in B$) then the first point in the orbit of $p$ not contained in $A$ ($B$) will be contained in a set $C_j$ ($j \neq i$), such that $j=\pi_R(i)$ ($j=\pi_L(i))$, where $\pi_L$ and $\pi_R$ are the $k$-permutations defined above.
\end{itemize}
Consider, for example, the point $p$ with index $p_A+1 \in A$. The sequence of sets visited by the orbit of $p$ is $(A,C_1,B,C_{\pi_R(1)},A,C_{\pi_L(\pi_R(1))},B,\ldots,B,C_{\pi_L^{-1}(1)}$. The orbit of $p$ contains all the points \emph{iff} the sequence contains each set $C_i$ exactly twice (one immediately after $A$ and one immediately after B). Since $\pi_L$ and $\pi_R$ are cyclic permutations, this is equivalent to each set $C_i$ appearing once after $A$. The sequence of indices for these sets is $(1,\pi_L\pi_R(1),(\pi_L\pi_R)^2(1),\ldots,(\pi_L\pi_R)^{-1}(1))$. This sequence contains all the indices $1,\ldots,k$ exactly once \emph{iff} $\pi_L\pi_R$ is a cyclic permutation. Finally, $\pi_R\pi_L$ is cyclic \emph{iff} $\pi_L\pi_R$ is cyclic (they are conjugate permutations), which completes the proof.
\end{proof}

Given three Lorenz braids $A$, $B$ and $C$, all of them closing to knots, let $\pi_R(A)$ be the $R$-permutation of $A$ and $\pi_L(B)$ be the $L$-permutation of $B$ as in Def. \ref{def:LRperm}. Assume that $\pi_R(A) \pi_L(B)$ is cyclic.

Let $m_R(i)$ be the number of $R$ symbols between the $i$-th and the $(i+1)$-th $L$ symbols in $B$ ($m_R(i)=0$ whenever there are two consecutive $L$ symbols), for $i=1,\dots,k-1$, and $m_R(k)$ the number of $R$ symbols after the last $L$. Likewise, let $m_L(i)$ the number of $L$ symbols between the $i$-th and the $(i+1)$-th $R$ symbols in $A$ ($m_L(i)=0$ if the $i$-th $R$ is immediately followed by another $R$), for $i=1,\dots,k-1$, $m_L(k)$ the number of $L$ symbols after the last $R$. For simplicity, we will write $\pi_R$ for $\pi_R(A)$ and $\pi_L$ for $\pi_L(B)$.

The following algorithm returns the aperiodic word $w(A,B,C)$ corresponding to the satellite braid constructed through the procedure above:

\begin{algorithm}\label{alg:satellite}

  \begin{itemize}
  \item[] 
  \item[] Input: Three aperiodic words $w_A$, $w_B$, $w_C$ with $n_R(A)=n_L(B)=k$, $w_A$ in $R$-minimal form, $w_B$ and $w_C$ in $L$-maximal form.
  
  \end{itemize}

\begin{enumerate}
\item For $i=1,\ldots,k$, $w_i=w_C$.
\item
  \begin{itemize}
  \item[] $j_1=k$.
  \item[] Insert  $R^{m_R(j_1)}$ in $w_1$ immediately after the maximal position;
  \item[] $j_2=\pi_R(j_1)$;
  \item[] Insert $L^{m_L(j_2)}$ in $w_1$ immediately after the minimal position.
\end{itemize}

\item For $i=2,\ldots,k$:
  \begin{itemize}
  \item[] $j_{2i-1}=\pi_L(j_{2i-2})$;
  \item[] Insert $R^{m_R(j_{2i-1})}$ in $w_i$ immediately after the maximal position of $w_i$;
  \item[] $j_{2i}=\pi_R(j_{2i-1})$;
  \item[] Insert $L^{m_L(j_{2i})}$ in $w_i$ immediately after the minimal position of $w_i$.
  \end{itemize}

\item $w_S=w_1w_2 \ldots w_k$

\item[] Output: Aperiodic word $w_S(A,B,C)$ of satellite in $L$-maximal form.

\end{enumerate}

\end{algorithm}

\begin{prop}\label{prop:satword}
  If the satellite braid constructed as above from Lorenz braids $A$, $B$ and $C$ represents a knot, then its aperiodic word is obtained by the preceding algorithm.
\end{prop}

\begin{proof}
  First note that the satellite braid has exactly $k\, n(C) + n_L(A) +n_R(B)$ strings, and its corresponding word $k\, t(C)$ syllables, (where $n(C)$ is the number of strings in $C$ and $t(C)$ is its trip number). Each string in $C$ originates $k$ strings in the satellite. To these, $n_L(A) +n_R(B)$ strings are added. The $LR$ strings in the satellite braid result from the inflation of the $LR$ strings in $C$. The $n_L(A)$ $L$ strings of $A$ and the $n_R(B)$ $R$ strings of $B$ contribute, respectively, with $n_L(A)$ $LL$ strings and $n_R(B)$ $RR$ strings. Therefore, the braid index of the satellite knot or link is $k\, t(C)$.

With the exception of the strings that result from the inflation of the first $RL$ string and the last $LR$ string in $C$ (connected, respectively, to the $R$ strings of $A$ and the  all strings in the inflated braid) are organized in sets of beams of $k$ parallel strings whose endpoints are the startpoints of another beam of $k$ parallel strings. The subwords in the word of the satellite braid will be the same as in the word for $C$ while the startpoints and endpoints of strings are within the inflated braid.

If we follow the satellite braid beginning with the maximal position of the last $LR$ beam, corresponding to an $L$ symbol both in $C$ and in the satellite word, then its endpoint will be the last in $B$, corresponding to the endpoint of the last $LR$ string of $B$. The $L$ symbol will thus be followed by $m_R(k)$ $R$ symbols, corresponding to $RR$ strings in $B$. When an $RL$ string of $B$ is reached, the endpoint is the  $R$ point in the last $RL$ beam of the satellite braid given by $j_2=\pi_L(k)$. Since the strings in the inflated braid are parallel, the next symbols in the satellite braid are the sequence between the maxima and minimal position of $C$. When a position in the first $RL$ braid is reached, $m_L(j_2)$ $L$ symbols are added in an analogous way and the inflated braid is resumed at the $\pi_R(j_2)$ and then continued by the symbols in $C$ until the the $\pi_R(j_2)$ position in the last beam. The process continues until all the points on the satellite braid have been visited once.
\end{proof}

\begin{example}
  For the Lorenz braid in Fig. \ref{fig:satbraid2}, $A=LRRLR$, $B=LRRLRLR$ and $C=LRRLR$ and $k=n_R(A)=n_L(B)=3$.

  In braid $B$, $l_1=LRLRLRR$, $l_2=LRLRRLR$, $l_3=LRRLRLR$, so $m_R(1)=1$, $m_R(2)=1$, $m_R(3)=2$ and, in cycle notation, $\pi_L=(1\,2\,3)$.

In braid $A$, $r_1=RLRLR$, $r_2=RLRRL$, $r_3=RRLRL$, $m_L(1)=1$, $m_L(2)=1$, $m_L(3)=0$ and $\pi_R=(1\,2\,3)$.

Starting with $w_1=w_2=w_3=LRRLR$ and following the steps above,
\begin{itemize}
\item add $m_R(3)=2$ $R$ symbols after the maximal position of $w_1$;
\item $\pi_R(3)=1$, so add $m_L(1)=1$ $L$ symbol to the minimal position of $w_1$;
\item $\pi_L(1)=2$, so add $m_R(2)=1$ $R$ to the maximal position of $w_2$;
\item $\pi_R(2)=3$, so add $m_L(3)=0$ $L$ to the minimal position of $w_2$;
\item $\pi_L(3)=1$, so add $m_R(1)=1$ $R$ at the maximal position of $w_3$;
\item $\pi_R(1)=2$, so add $m_L(2)=1$ $L$ to the minimal position of $w^3$.
\end{itemize}
 The resulting aperiodic word $w_S=w_1w_2w_3=LRRRRLLRLRRRLRLRRRLLR$ corresponds to the satellite braid in Fig. \ref{fig:satbraid2}.
\end{example}

\begin{theorem}\label{theor:permnotsat}
  No permuted words in the sets $P(p,q)$ can be obtained through the satellite braid construction defined above.
\end{theorem}
\begin{proof}
  As the words in $P(p,q)$ have no $LL$ subwords, they would have to be obtained, by the procedure above, from the word of a companion knot $C$ which equally has no $LL$ subsequences, taking $A$ as an identity braid and a Lorenz braid $B$. The word for $C$ would have the structure (in $L-$maximal form) $LR^{b_1}\dots LR^{b_a}$ where $a=n_L(C)>1$ (otherwise $C$ would be trivial), $b_1=\max\{b_i\}$ and at least one $k$ satisfying $b_k < b_1$ (otherwise the word would be periodic). Each word in $P(p,q)$ has syllables of only two types: $LR^{k+1}$ and $LR^k$. The satellite construction, in this case, only adds $R$ symbols to the maximal positions of $w_1,\ldots,w_k$, so if it could be obtained through this procedure then we would have $\min\{b_i\}=k$ and $b_1=k+1$. But after the satellite construction, the resulting word would have a first syllable with at least $b_1+1=k+2$ symbols $R$, and at least one syllable with $k$ symbols $R$ and thus will not be in any $P(p,q)$.
\end{proof}

El-Rifai showed \cite{Elrifai99} that a Lorenz knot that is a satellite of a Lorenz knot can be presented as the closure of a Lorenz braid constructed as outlined above. Morton has conjectured \cite{Elrifai88},\cite{Dehornoy11} that all Lorenz satellite knots are cablings (satellites where the pattern is a torus knot) on Lorenz knots.

If Morton's conjecture is true, then we conclude from Thm.\ref{theor:permnotsat} that Lorenz knots corresponding to syllable permutations of standard torus words, that is, the knots corresponding to words in the sets $P(p,q)$ are not satellites.
  
\section{Syllable permutations that do not correspond to torus knots}
\label{sec:Lorenztor}

As seen above, the braid index of any braid corresponding to a word in $P(p,q)$ is given by $t=n_L=p$. So, if any ot these knots were torus knots, then they would have to be of type $T(p,q')$ for some $q'$.

As a consequence of remarks \ref{rem:p>4} and \ref{rem:r=1,p-1} in Sec. \ref{sec:torsylperm}, we again assume $p>4$, $1<r<p-1$ below.

\begin{lem}\label{lem:permnotorus1}
  No closure of a braid corresponding to a word in $P(p,q)$ is a torus knot $T(p,q')$, for $q'>q$, and the only word in $P(p,q)$ corresponding to a braid that has $T(p,q)$ as closure is the evenly distributed word $w(p,q)$ corresponding to the standard $T(p,q)$ braid.
\end{lem}
\begin{proof}
  If a word in $P(p,q)$ corresponds to a knot $T(p,q')$, then the genus of $T(p,q')$ must be in the range defined by Eq. \ref{eq:genus}, so in particular, $(p-1)(q'-1) \leq (p-1)(q-1) \Rightarrow q' \leq q$. The only word in $P(p,q)$ for which the genus is $(p-1)(q-1)$ is the word corresponding to the standard braid of $T(p,q)$.
\end{proof}

\begin{theorem}\label{theor:unique}
  For each word $w$ in $P(p,q)$, $4<p<q$, distinct from $w(p,q)$, there is at most one torus knot $T(p,q'),\ q'<q$, with the same braid index and genus as the closure of the braid corresponding to $w$.
\end{theorem}
\begin{proof}
  Let $g$ be the genus of the closure of the Lorenz braid corresponding to $w$. As $w \in P(p,q)$, any torus knot with the same braid index must be of type $T(p,q')$ for some $q'<q$. The genus of $T(p,q')$ is given by $(p-1)(q'-1)$ so $q'$ is uniquely determined by $(p-1)(q'-1)=2g$. If $p-1$ divides $2g$ then this equality determines a unique $q'$, otherwise there is no $q'$ satisfying the condition.
\end{proof}

\begin{lem}\label{lem:kp}
  Let $w \in P(p,q)=P(p,kp+r)$, $p>4$, $1<r<p-1$. If there is a (unique) torus knot $T(p,q')=T(p,k'p+r')$ with the same genus and braid index as the closure of the braid of $w$, then $k'=k$ and $1<r'<r<p-1$.
\end{lem}
\begin{proof}
  If $k'>k$ then $k \geq k+1$, so $q'=k'p+r' \geq (k+1)p+r'=kp+p+r'>kp+r=q$. But $q'<q$ (Lemma \ref{lem:permnotorus1}) so $k'\leq k$.

From the first inequality in Eq. \ref{eq:genus},
$kp(p-1)+r(r-1) \leq (p-1)(k'p+r'-1)$ or $$\left((k-k')p-r'+1 \right)(p-1)+r(r-1)\leq 0.$$
If $k'<k$ then $(k-k')p \geq p$ so $(k-k')p -r'+1 \geq p-r'+1 \geq 0$ and the inequality above doesn't hold.

For $k=k'$, $q'<q \Rightarrow r'<r$. If $r'=1$ then the same inequality gives $r(r-1)<0$ (impossible) which completes the proof.

\end{proof}

\begin{prop}\label{prop:notorus1}
  For each odd integer $p>4$ and all integer $k>0$, the sets $P(p,q)$, for $q=kp+2$ and $q=(k+1)p-2$ contain no words corresponding to braids whose closure is a torus knot, besides $w(p,q)$.
\end{prop}

\begin{proof}
  For all odd positive integers $p$ and all $k>0$, $(p,kp+2)$ and $(p,(k+1)p-2)$ are pairs of relatively prime integers. If $q=kp+2$, then $r=2$ and there is no $r'$ satisfying the condition of Lem. \ref{lem:kp}.
  For $q=(k+1)p-2$, from the proof of Theorem \ref{theor:unique}, if there is a word in $P(p,q)$ corresponding to a braid whose closure is a torus knot of genus $g$, then $(p-1)$ divides $2g$. As the maximum $2g$ in $P(p,q)$ is $(p-1)(q-1)$, there will be at least another (even) multiple of $(p-1)$ in the set $\{kp(p-1)+r(r-1),\ldots,(p-1)(kp+r-1)\}$ if $kp(p-1)+r(r-1) \leq (p-1)(kp+r-2)$ or
$$ r(r-1) \leq (p-1)(r-2).$$
This inequality above does not hold, with any $p>4$, for $r=p-2$, as can be seen from simple substitution, which completes the proof.
\end{proof}

\begin{prop}\label{prop:notorus2}
  If $p>4$ is even and not a multiple of $3$, then for any integer $k$ the sets $P(p,kp+3)$ and $P(p,(k+1)p-3)$ contain no words corresponding to braids whose closure is a torus knot. Also, if $p<12$ and $p$ is even, then $P(p,q)$ contains no words other than $w(p,q)$ corresponding to torus knots.
\end{prop}
\begin{proof}
  If $p$ is even and not a multiple of $3$ then $(p,kp+3)$ and $(p,(k+1)p-3)$ are pairs of relatively prime integers. Also, since $p-1$ is odd, $(p-1)(kp+r-1)$ is an even multiple of $p-1$ (because it is equal to $2g$ for the torus knot $T(p,kp+r)$, so $(p-1)(kp+r-2)$ is odd. There will be at least another even multiple of $(p-1)$ in the set $\{kp(p-1)+r(r-1),\ldots,(p-1)(kp+r-1)\}$ if $kp(p-1)+r(r-1) \leq (p-1)(kp+r-3)$. Direct substitution  shows this inequality does not hold for $r=3$ and $r=p-3$ or, for any $1<r<p-1$, for $p<12$ even.
\end{proof}
\begin{rem}
  If $p$ is even, then $(p,kp+2)$ and$(p,(k+1)p-2)$ are not pairs of relatively prime integers. The same happens with $(p,kp+3)$ and $(p,(k+1)p-3)$ if $p$ is a multiple of $3$.
\end{rem}

\section{Conclusion}

\label{sec:concl}
\begin{itemize}
\item We have defined sets $P(p,q)$ of words resulting from permuting the syllables of the standard words for torus knots $T(p,q), 1<p<q$ and characterized their corresponding Lorenz braids. For $p \leq 4$ $P(p,q)=\{w(p,q)\}$. We have shown that for each word $w$ in the sets $P(p,q)$ there is at most a torus knot $T(p,q')$ with the same braid index and genus as the knot coresponding to $w$. 
\item A procedure to obtain Lorenz satellite braids was defined, and an algorithm for finding its aperiodic word was devised. We have used this algorithm to prove that, provided Morton's conjecture is true, no set $P(p,q)$ contains words corresponding to satellite knots.
\item From Theorem \ref{theor:permnotsat}, Propositions \ref{prop:notorus1} and \ref{prop:notorus2} and Thurston's theorem \cite{Thurston82} we can finally state that:

If Morton's conjecture is true, then the words in sets $P(p,kp+2)$, $P(p,(k+1)p-2)$ for $p$ odd and $P(p,kp+3)$, $P((k+1)p-3)$, for $p$ even and $p$ not a multiple of $3$, distinct from the standard $w(p,q)$ torus word, correspond to hyperbolic Lorenz knots.
\item We have recently performed an extensive computational test \cite{Gomes14}, in which we computed the volumes of all knot complements corresponding to words in the (non-empty) sets $P(p,q)$ with $5\leq p \leq 19$ and $6 \leq q \leq 100$. We found all of them to be hyperbolic, with the expected exception of the torus knots $T(p,q)$ corresponding to the standard words $w(p,q)$. This allowed us to conjecture that the results of this work can be extended to all $P(p,q)$ sets.
\item We believe the tecniques used in this paper can be generalized to study other families of words and their corresponding Lorenz braids and knots and thus possibly generate other infinite families of hyperbolic Lorenz knots.
\end{itemize}

\paragraph{Acknowledgements}
\begin{itemize}
\item[] The authors would like to thank Juan Gonz\'alez-Meneses and Hugh Morton for sharing their knowledge and valuable insight during the preparation of this paper.
\item[] Research partially funded by FCT/Portugal through project PEst-OE/\allowbreak{}MAT/\allowbreak{}UI0117/2014. 
\end{itemize}

\end{document}